\documentclass[12pt]{amsart}
\usepackage{amsmath}
\usepackage{amsthm}
\usepackage{enumerate}
\usepackage{todonotes}
\usepackage[pagebackref,bookmarksopen]{hyperref}
\usepackage{comment}
\hypersetup{
	colorlinks=true,
	linkcolor=blue,
	citecolor=magenta,
	urlcolor=cyan,
	}

\newtheorem*{main}{Theorem}
\newtheorem{theo}{Theorem}[section]
\newtheorem{lemma}[theo]{Lemma}
\newtheorem{cor}[theo]{Corollary}
\newtheorem{prop}[theo]{Proposition}
\theoremstyle{definition}
\newtheorem{defi}[theo]{Definition}
\newtheorem{remark}[theo]{Remark}
\newtheorem{example}[theo]{Example}
\newcommand{\e}{\varepsilon}
\newcommand{\B}{{\mathbb{B}}}
\newcommand{\C}{{\mathbb{C}}}
\newcommand{\R}{{\mathbb{R}}}
\newcommand{\Rp}{\underline{\mathbb{R}}}
\newcommand{\BS}{\mathbb{S}}
\newcommand{\CC}{\mathcal{C}}
\newcommand{\LL}{\mathcal{L}}
\newcommand{\Z}{{\mathbb{Z}}}

\begin{document}
\title[]{Normalized Milnor Fibrations for Real Analytic Maps}

\author[J.~L.~Cisneros-Molina]{Jos\'e Luis Cisneros-Molina}
\address{Instituto de Matem\'aticas, Unidad Cuernavaca\\ Universidad Nacional Aut\'onoma de M\'exico\\ Avenida Universidad s/n, Colonia Lomas
de Chamilpa\\ Cuernavaca, Morelos, Mexico.}
\curraddr{}
\email{jlcisneros@im.unam.mx}

\author[A.~Menegon]{Aur\'elio Menegon}
\address{Mid Sweden University \\ Department of Engineering, Mathematics and Science Education \\ 85170 Sundsvall, Sweden}
\curraddr{}
\email{aurelio.menegon@miun.se}

\begin{abstract}
Milnor’s fibration theorem and its generalizations play a central role in the
study of singularities of complex and real analytic maps.
In the complex analytic case, the Milnor fibration on the sphere is always
given by the normalized map $f/|f|$.
In contrast, for real analytic maps the existence of such a normalized Milnor
fibration generally fails, even when a Milnor--L\^e fibration exists on a tube.

For locally surjective real analytic maps $f \colon (\mathbb{R}^n,0)\to
(\mathbb{R}^k,0)$ with isolated critical value, the existence of a Milnor--L\^e
fibration on a tube is guaranteed under a transversality condition.
However, the associated fibration on the sphere need not be given by the
normalized map $f/\|f\|$, unless an additional regularity condition
($d$-regularity) is imposed.

In this paper we show that this apparent obstruction is not intrinsic.
More precisely, we prove that for any such map satisfying the transversality
property, there exists a homeomorphism $h\colon (\mathbb{R}^k,0) \to (\mathbb{R}^k,0)$
of the target space such that the composition $h^{-1}\circ f$ becomes
$d$-regular.
As a consequence, the normalized map
\[
\frac{h^{-1}\circ f}{\|h^{-1}\circ f\|}\colon
\mathbb{S}^{n-1}_\varepsilon\setminus f^{-1}(0)\longrightarrow \mathbb{S}^{k-1}
\]
defines a smooth locally trivial fibration.
This fibration is equivalent to both the Milnor--L\^e fibration on the tube
and the Milnor fibration on the sphere.

Our result reveals a closer topological parallel between real and complex analytic singularities
than previously recognized, without changing the topological type of the
singularity.
\end{abstract}

\maketitle

\section{Introduction}\label{sec:introduction}

Milnor’s fibration theorem plays a central role in the study of singularities of complex analytic functions, see for instance \cite{Cisneros-Seade:MFTRCS} for an overview of its origin, generalizations and connections with other bran\-ches of mathematics.
To each holomorphic function $f\colon(\C^n,0)\to(\C,0)$ with a critical point at the origin $0\in\C^n$, it associates a locally trivial fibration \cite[Theorem~4.8]{Milnor:SPCH}, called the \textit{Milnor fibration}, given by
\begin{equation}\label{eq:CMF}
\frac{f}{\vert f\vert}\colon \BS_\varepsilon\setminus f^{-1}(0)\to\BS^1,
\end{equation}
where $\BS_\varepsilon$ is a sufficiently small sphere around $0\in\C^n$.

There is a second locally trivial fibration, called the \textit{Milnor-Lê fibration}, its total space is a \textit{Milnor tube}
\begin{equation*}
T_{\varepsilon,\delta}(f):=\B_\varepsilon\cap f^{-1}(\BS^1_\delta),
\end{equation*}
where $\B_\varepsilon$ is the ball bounded by the sphere $\BS_\varepsilon$ and $\BS^1_\delta$ is the circle of radius $\delta$ centered at the origin of $\C$, and its base space is $\BS^1_\delta$, that is, when $0<\delta\ll\varepsilon$ the restriction
\begin{equation}\label{eq:CMLF}
 f\colon T_{\varepsilon,\delta}(f)\to\BS^1_\delta
\end{equation}
is a locally trivial fibration. Its existence was proved by Milnor \cite[Theorem~11.2]{Milnor:SPCH} when $f$ has an isolated critical point at $0\in\C^n$ and by Lê in \cite[Theorem~(1.1)]{Le:SRRM} for the general case, both using Ehresmann Fibration Theorem for manifolds with boundary.

Fibrations \eqref{eq:CMF} and \eqref{eq:CMLF} are equivalent, this is proved using a smooth vector ﬁeld on $\B_\varepsilon\setminus f^{-1}(0)$ constructed by Milnor in \cite[Lemma~5.4]{Milnor:SPCH} to ``inflate'' the Milnor tube to the sphere, i.e. following the ﬂow of the vector ﬁeld one gets a diﬀeomorphism between the Milnor tube and the complement on the sphere of a neighborhood of the link.
This vector field has the following properties:
\begin{enumerate}[1.]
 \item It is transverse to the spheres centered at the origin contained in $\B_\varepsilon$;\label{it:ts}
 \item It is transverse to the Milnor tubes;\label{it:tt}
 \item Given any integral curve $p(t)$ of such vector field $\frac{f(p(t))}{\Vert f(p(t))\Vert}$ is a constant point $\theta\in\BS^1\subset\C$, that is, $p(t)$ is projected by $f$ on the ray $\LL_\theta$ from $0$ passing through $\theta\in\BS^1$ in $\C$.\label{it:ca}
\end{enumerate}
See \cite{Aguilar-Cisneros:GVMFT} for the proofs of Milnor Fibration Theorem, Milnor-Lê Fibration Theorem and the equivalence of these two fibrations, emphasizing the ideas of differential topology involved.

Milnor also proved Fibration Theorems for real analytic maps
\begin{equation*}
 f \colon (\mathbb{R}^n,0) \to (\mathbb{R}^k,0), \qquad 2 \le k \le n,
\end{equation*}
with \textit{isolated critical point} at $0\in\R^n$. First he proved the existence of a locally trivial fibration on a Milnor tube $T_{\varepsilon,\delta}(f):=\B_\varepsilon^n\cap f^{-1}(\BS_\delta^{k-1})$
(Minor-Lê fibration)\cite[Theorem~2]{Milnor:ISH} or \cite[Theorem~11.2]{Milnor:SPCH}
\begin{equation}\label{eq:RMLF}
f\colon T_{\varepsilon,\delta}(f)\to\BS_\delta^{k-1},
\end{equation}
where $\BS_\delta^{k-1}$ is the sphere of radius $\delta$ centered at $0\in\R^k$ and $0<\delta\ll\varepsilon$. Then he proved that on the sphere $\BS^{n-1}_\varepsilon=\partial\B_\varepsilon^n$ there is a locally trivial fibration
\begin{equation}\label{eq:RMF}
\varphi_M \colon \BS^{n-1}_\varepsilon \setminus f^{-1}(0) \longrightarrow \BS^{k-1},
\end{equation}
constructing a vector field on $\B_\varepsilon^n\setminus f^{-1}(0)$ which is transverse to the spheres centered at the origin contained in $\B_\varepsilon^n$ and it is transverse to the Milnor tubes,
to “inﬂate” the Milnor tube to the sphere \cite[Lemma~11.3]{Milnor:SPCH}. These fibration theorems have some weaknesses: the maps $f$ with isolated critical point are highly non-generic, so it is difficult to find examples, the projection $\varphi_M$ cannot be explicitly described and in general it is not true that the map
\begin{equation}\label{eq:SMF}
\frac{f}{\Vert f\Vert}\colon \BS^{n-1}_\varepsilon \setminus f^{-1}(0) \longrightarrow \BS^{k-1}
\end{equation}
is the projection of a locally trivial fibration \cite[page~99]{Milnor:SPCH} because in general it is not a submersion, which is a necessary condition to be a smooth ﬁber bundle.
When the map \eqref{eq:SMF} is a locally trivial fibration, in \cite{Ruas-Seade-Verjovsky:RSMF} the authors said that $f$ satisfies the \textit{strong Milnor condition}, they
gave examples of these maps and study the stability of this condition under perturbations of $f$. When \eqref{eq:SMF} is a locally trivial fibration we call it a \textit{normalized Milnor fibration}.

In \cite{Cisneros-Seade-Snoussi:d-regular} the Milnor-Lê fibration \eqref{eq:RMLF} was generalized when $f \colon (\mathbb{R}^n,0) \to (\mathbb{R}^k,0)$,  with $2 \leq k \leq n$, is a locally surjective real analytic map with an \textit{isolated critical value} at $0\in\R^k$, satisfying the following \textit{transversality property:} there exist a \textit{solid Milnor tube}
$\widehat{T}_{\varepsilon,\delta}(f):=\B_\varepsilon^n\cap f^{-1}(\B_\delta^k)$, with $0<\delta\ll\varepsilon$, such that all the ﬁbers in the tube are transverse to $\BS_\varepsilon^{n-1}$, to be able to apply Ehresmann Fibration Theorem for manifolds with boundary \cite[Theorem~5.1, Remark~5.7]{Cisneros-Seade-Snoussi:d-regular}.
Having the ﬁbration on the tube, one can use Milnor’s vector ﬁeld to inﬂate the tube to the sphere to get a ﬁbration \eqref{eq:RMF} on the sphere, but again, not necessarily with projection $\frac{f}{\Vert f\Vert}$.
In \cite[Definition~2.4]{Cisneros-Seade-Snoussi:d-regular} the concept of \textit{$d$-regularity} was introduced, and it was proved \cite[Proposition 3.2(4)]{Cisneros-Seade-Snoussi:d-regular}
that $f$ is $d$-regular if and only if the map \eqref{eq:SMF} is a submersion, so $d$-regularity is a necessary condition for \eqref{eq:SMF} to be a ﬁber bundle. On the other hand in
\cite{Cisneros-Seade-Snoussi:d-regular,Cisneros-Menegon:EMFMLF,Cisneros-Menegon:EEMMLF} it was also proved that $f$ is $d$-regular if and only if there exist a vector field in $\B_\varepsilon^n\setminus f^{-1}(0)$ which satisfies properties \ref{it:ts}, \ref{it:tt} and \ref{it:ca}, which implies that ``inﬂating'' the Milnor tube to the sphere, the map \eqref{eq:SMF} is a ﬁber bundle equivalent to the ﬁber bundle \eqref{eq:RMLF}, so $d$-regularity is a necessary and suﬃcient condition for $f$ to have a normalized Milnor fibration \eqref{eq:SMF}. In this case we also have fibration \eqref{eq:RMF} which it is equivalent to fibration \eqref{eq:SMF}, since both are equivalent to fibration \eqref{eq:RMLF}.

However, many locally surjective real analytic maps with isolated critical value satisfying the transversality property fail to be $d$-regular.
The main purpose of this paper is to show that this apparent obstruction can always be removed by a suitable change of coordinates in the target space.
More precisely, there \textit{always} exists a suitable homeomorphism $h$ of the target such that the composition $f_h := h^{-1}\circ f$ is $d$-regular, i.~e., it has a normalized Milnor fibration given by
\begin{equation}\label{eq:hSMF}
\frac{f_h}{\Vert f_h\Vert}\colon \BS^{n-1}_\varepsilon \setminus f^{-1}(0) \longrightarrow \BS^{k-1}.
\end{equation}

One of the surprising aspects of the present work is that it reveals a much closer relationship between the real and the complex settings than previously expected.
In the complex analytic case, Milnor showed that the fibration \eqref{eq:CMF} on the sphere is always given by the normalized map $f/\vert f\vert$.
In contrast, in the real analytic case it has long been understood that such a simple description generally fails, unless strong regularity conditions are imposed.
Our results show that this difference is, in fact, less intrinsic than it might appear.

It is worth emphasizing that this modification does not alter the singularity type of the map, as the map $f_h$ is topologically $\mathcal{A}$-equivalent to $f$, being obtained by composition with a homeomorphism of the target.

The main novelty of this work is to show that the obstruction to normalized
Milnor fibrations in the real analytic setting is not intrinsic, but can
always be removed by a suitable topological change of coordinates in the
target. From a topological viewpoint, this shows that the failure of
normalized Milnor fibrations for real analytic maps is not a genuine
obstruction, but rather an artifact of the chosen target coordinates.
As a consequence, real analytic Milnor fibrations turn out to be much
closer to the complex case than previously expected. To our knowledge,
this provides the first general result ensuring the existence of normalized
Milnor fibrations for real analytic maps with isolated critical value.

More precisely, we prove the following result.

\begin{main} \label{principal}
Let $f \colon (\mathbb{R}^n,0) \to (\mathbb{R}^k,0)$, with $2 \le k \le n$, be a locally surjective real analytic map with an isolated critical value at $0$.
Assume moreover that $f$ satisfies the transversality property. Then there exists a homeomorphism
\begin{equation*}
h \colon (\mathbb{R}^k,0) \to (\mathbb{R}^k,0)
\end{equation*}
such that the map
\begin{equation*}
 \frac{h^{-1}\circ f}{\|h^{-1}\circ f\|}\colon\BS^{n-1}_\varepsilon \setminus f^{-1}(0) \longrightarrow \BS^{k-1},
\end{equation*}
is a smooth locally trivial fibration. Moreover, this fibration is equivalent to both the Milnor--L\^e fibration on the tube \eqref{eq:RMLF} and the Milnor fibration on the sphere \eqref{eq:RMF}.
\end{main}

In the terminology of \cite{Cisneros-etal:FTDRDMGNICV-pp}, Theorem~\ref{principal} means that every analytic map-germ as above is {\it $d_h$-regular} for a suitable {\it conic homeomorphism $h$}. In particular, it provides a positive answer to \cite[Question~3.10]{Cisneros-etal:FTDRDMGNICV-pp}.

\medskip
\section{The $d_h$-regularity for analytic maps with isolated critical value}  
\label{section_3}

In this section, we recall some definitions and results from \cite{Cisneros-etal:FTDRDMGNICV-pp} that extend the concept of $d$-regularity, allowing some maps to become $d$-regular after a homeomorphism on the target space. Then we apply those results to real analytic maps with an isolated critical value.

Given $\eta>0$, for each point $\theta \in \BS_\eta^{k-1}$ the set $\LL_\theta \subset \R^k$ is the open segment of line that starts in the origin and ends at the point $\theta$ (but not containing these two points).

We say that a homeomorphism
\[
 h\colon (\R^k,0) \longrightarrow (\R^k,0)
\]
is a \textit{conic homeomorphism} if there exists $\eta_0>0$ such that for every $\eta$ with $0<\eta<\eta_0$ the restriction
\[
 h\colon \B_\eta^k \to h(\B_\eta^k) \, ,
\]
satisfies the following:
\begin{itemize}
\item[$(i)$] For each $\theta \in \BS_\eta^{k-1}$ the image $h(\LL_\theta)$ is a smooth path in $\R^k$;
\item[$(ii)$] The inverse map $h^{-1}$ of $h$ is smooth outside the origin;
\item[$(iii)$] The map $h^{-1}$ is a submersion outside the origin.
\end{itemize}
To simplify the notation, we set $\mathcal{B}_\eta^k := h(\B_\eta^k)$, $\mathcal{S}_\eta^k := h(\BS_\eta^k)$,
\begin{equation*}
 \CC_\theta := h(\LL_\theta),\quad\text{and}\quad  E_{\theta} := f^{-1}(\CC_\theta)
\end{equation*}
for each $\theta \in \BS_\eta^{k-1}$.

Now let $f\colon (\R^n,0) \to (\R^k,0)$, with $2 \leq k \leq n$, be a locally surjective analytic map-germ with isolated critical value and the transversality property.
By \cite[Theorem~5.1, Remark~5.7]{Cisneros-Seade-Snoussi:d-regular} there exist $0<\delta\ll\varepsilon$ such that there is a Milnor-Lê fibration on a Milnor tube
\begin{equation}\label{eq:RMLF2}
f\colon T_{\varepsilon,\delta}(f)\to\BS_\delta^{k-1}.
\end{equation}
Fix the Milnor tube $T_{\varepsilon,\delta}(f)$ and denote by $\mathring{T}_{\varepsilon,\delta}(f):=\B_\varepsilon^n\cap f^{-1}(\mathring{\B}_\delta^k) $ the \textit{open solid Milnor tube}.
Let $h\colon \B_\eta^k \to\mathcal{B}_\eta^k$ be a conic homeomorphism and set $f_h:=h^{-1}\circ f$.
We say that $f$ is \textit{$d_h$-regular} (outside $\mathring{T}_{\varepsilon,\delta}(f)$) if $\varepsilon >0$ is small enough (with $f(\B_{\varepsilon}^n) \subset \mathring{\mathcal{B}}_{\eta_0}^k$) such that for every $\varepsilon'$ with $0<\varepsilon' \leq \varepsilon$ the sphere $\BS_{\varepsilon'}^{n-1}$ intersects $E_{\theta}$ transversely (whenever such intersection is not empty) for each $\theta\in\BS_\eta^{k-1}$ at every point $p\in\B_\varepsilon^n\setminus\mathring{T}_{\varepsilon,\delta}(f)$.

In this case, by \cite[Theorem~3.12]{Cisneros-etal:FTDRDMGNICV-pp} the map
\begin{equation}\label{eq:fh-fib}
\frac{f_h}{\|f_h\|}\colon \BS_\varepsilon^{n-1} \setminus f^{-1}(0)  \to \BS^{k-1}
\end{equation}
is the projection of a smooth locally trivial fibration. 

\begin{remark}
Note that the previous definition of $d_h$-regularity is a bit weaker than the definition given in \cite[\S 3.2]{Cisneros-etal:FTDRDMGNICV-pp} since we only ask the transversality of $E_\theta$ with the spheres \textit{outside} of the open solid Milnor tube $\mathring{T}_{\varepsilon,\delta}(f)$, but this is enough to prove the equivalence of the Milnor-Lê fibration \eqref{eq:RMLF2} and the normalized Milnor fibration \eqref{eq:fh-fib} \cite[Theorem~2.16]{Cisneros-etal:FTDRDMGNICV-pp}, since we only need the vector field defined outside the open solid Milnor tube, to ``inflate'' the Milnor tube to the sphere. In fact, we can take a slightly smaller open solid Milnor tube so that the boundary points of the Milnor tube $T_{\varepsilon,\delta}(f)$ will not cause any trouble to define the vector field.
\end{remark}

Let us show now that fibration \eqref{eq:fh-fib} is equivalent to fibration \eqref{eq:RMLF} and also equivalent to fibration \eqref{eq:RMF} above.
In fact, it follows from (\cite[Theorem~2.16]{Cisneros-etal:FTDRDMGNICV-pp}) that \eqref{eq:fh-fib} is equivalent to the fibration
\begin{equation*}
f_h\colon \B_\varepsilon^n \cap f_h^{-1}(\BS_{\delta'}^{k-1}) \to \BS_{\delta'}^{k-1} \, . \end{equation*}
But this last one is clearly equivalent to the fibration
\begin{equation} \label{eq:fib}
 f\colon \B_\varepsilon^n \cap f^{-1}(\mathcal{S}_{\delta'}^{k-1}) \to \mathcal{S}_{\delta'}^{k-1}  \, .
\end{equation}
So we just have to show that fibration (\ref{eq:fib}) is equivalent to the fibration 
\begin{equation} \label{eq:fib2}
 f\colon \B_\varepsilon^n \cap f^{-1}(\BS_{\delta''}^{k-1}) \to \BS_{\delta''}^{k-1}
\end{equation}
for some sphere sufficiently small such that $\BS_{\delta''}^{k-1}\subset\B_\delta^k$. But this follows from the next lemma using the Milnor-Lê fibration on a solid Milnor tube
\cite[Proposition~5.1]{Cisneros-Seade-Snoussi:d-regular}:

\begin{lemma}
Consider the fibration on the solid tube induced by $f$
\begin{equation}
 f|\colon \B_\varepsilon^n \cap f^{-1}(\mathring{\B}_\delta^k\setminus\{0\}) \to (\mathring{\B}_\delta^k\setminus\{0\}) \, , \label{eq:SMT}
\end{equation}
and notice that $\pi_{k-1}(\mathring{\B}_\delta^k\setminus\{0\})\cong\Z$. Let $\tilde{g},\tilde{h}\colon \BS^{k-1}\to \mathring{\B}_\delta^k\setminus\{0\}$ be two continuous embeddings such that both 
represent the generator of  $\pi_{k-1}(\mathring{\B}_\delta^k\setminus\{0\})$. Then the restrictions
\begin{align}
f|\colon \B_\varepsilon^n \cap f^{-1}\bigl(g(\BS^{k-1})\bigr) &\to g(\BS^{k-1}) \label{eq:fib.g}\\
\intertext{and}
f|\colon \B_\varepsilon^n \cap f^{-1}\bigl(h(\BS^{k-1})\bigr) &\to h(\BS^{k-1})\label{eq:fib.h}
\end{align}
are equivalent fibrations.
\end{lemma}

\begin{proof}
For paracompact spaces locally trivial fibrations are homotopy invariant, that is, pull-backs over homotopic maps give equivalent fibre bundles (see for instance \cite[\S4 
Theorem~9.8]{Husemoller:Bundles} plus a reduction to the principal bundle case). Since $\tilde{g}$ and $\tilde{h}$ are homotopic, the pull-backs of the fibre bundle \eqref{eq:SMT} by $\tilde{g}$ and 
$\tilde{h}$ are equivalent. Since $\tilde{g}$ and $\tilde{h}$ are embeddings those pull-backs are equivalent to the restrictions \eqref{eq:fib.g} and \eqref{eq:fib.h}.
\end{proof}

Thus we have proved the following:

\begin{prop} \label{prop_principal}
Let $f\colon (\R^n,0) \to (\R^k,0)$ be a locally surjective analytic map with an isolated critical value satisfying the transversality property.
If $h\colon (\R^k,0) \to (\R^k,0)$ is a conic homeomorphism such that $f$ is $d_h$-regular, then the map
\begin{equation}\label{SF}
\frac{h^{-1} \circ f}{\|h^{-1} \circ f\|}\colon\BS_\varepsilon^{n-1} \setminus f^{-1}(0)  \to \BS^{k-1}
\end{equation}
is the projection of a smooth locally trivial fibration. Moreover, this fibration is equivalent to fibration \eqref{eq:RMLF} and to fibration \eqref{eq:RMF}.
\end{prop}

As an immediate consequence, we have:

\begin{cor} \label{cor_7}
The locally trivial fibration \eqref{SF} of Proposition~\ref{prop_principal} above does not depend on the choice of the conic homeomorphism $h$ (as long as $f$ is $d_h$-regular), up to topological equivalence of fiber
bundles.
\end{cor}

This means that we are not going to find different fibrations on the sphere by looking for different conic homeomorphisms $h$ such that $f$ is $d_h$-regular.

We finish this section with an example:

\begin{example} \label{ex_1}
Consider the map $f: (\R^3,0) \to (\R^2,0)$ given by 
$$f(x,y,z)=(x^2z+y^3-z,x) \, .$$
It has an isolated critical point, so it has the transversality property. Using Proposition 3.8 of \cite{Cisneros-Seade-Snoussi:RSCG}, one can check that $f$ is not $d$-regular. In fact, the matrix $M$ (as defined in that Proposition) has rank less than $2$ at every point of the form $(x,0,0)$. Hence the map $\BS_\varepsilon^{n-1} \backslash f^{-1}(0) \to \BS^1$ given by $\phi(x) := \frac{f}{\|f\|}$ is not the projection of a fibration equivalent to the Milnor-L\^e fibration on the tube.

\noindent Given $\eta>0$, consider the conic homeomorphism $h: \B_\eta^2 \to \B_\eta^2$ given by
$$
h(y_1,y_2) =
\begin{cases}
\eta
\begin{bmatrix}
\frac{1}{\sqrt{1+(y_2/y_1)^2}} \exp \left( 1- \frac{\eta}{\sqrt{y_1^2+y_2^2}} \right) \\ 
\frac{y_2/y_1}{\sqrt{1+(y_2/y_1)^2}} \exp \left( \frac{1}{2}- \frac{\eta}{2\sqrt{y_1^2+y_2^2}} \right) \\
\end{bmatrix}
& \text{if} \ y_1 \neq 0 \\
\\
\eta
\begin{bmatrix}
\frac{y_1/y_2}{\sqrt{1+(y_1/y_2)^2}}  \exp \left( 1- \frac{\eta}{\sqrt{y_1^2+y_2^2}} \right) \\ 
\frac{1}{\sqrt{1+(y_1/y_2)^2}} \exp \left( \frac{1}{2}- \frac{\eta}{2\sqrt{y_1^2+y_2^2}} \right) \\
\end{bmatrix}
& \text{if} \ y_2 \neq 0 \\
\end{cases}
$$
Its inverse $h^{-1}: \B_\eta^2 \to \B_\eta^2$ is given by
$$h^{-1}(y_1,y_2) = \xi(y_1,y_2) \cdot \left( 2\eta y_1,  y_2 \sqrt{2y_2^2 +2 \sqrt{y_2^4+4y_1^2\eta^2}}  \right) $$
where 
$$\xi(y_1,y_2):= \frac{\eta}{\left[1-\ln(\frac{y_2^2+ \sqrt{y_2^4 + 4y_1^2 \eta^2}}{2 \eta^2})\right] \left[ y_2^2+\sqrt{y_2^4 + 4y_1^2 \eta^2} \right]} \, .$$
\end{example}

\section{Conic vector fields}

The goal of this section is to construct a family of conic homeomorphisms of the target space via flows of suitable vector fields.

We start defining some vector fields $\vec{v}_\alpha\colon\R^k\to\R^k$ on $\R^k$ parametrized by vectors $\alpha\in\R^k$.
The vector fields $\vec{v}_\alpha$ are perturbations of the radial vector field in $\R^k$ and we shall see that they give rise to conic homeomorphisms.

\begin{remark}
    To distinguish the space of parameters $\R^k$ from the space $\R^k$ where the vector field $\vec{v}_\alpha$ are defined, we underline the parameter space $\Rp^k$ or
any ball $\underline{\mathring{\B}}_r^k\subset\Rp^k$ in it.
\end{remark}

\begin{defi}
For each point $\alpha\in \Rp^k$ define the vector field in $\R^k$ given by
\begin{equation}\label{eq:source.vf}
\vec{v}_\alpha(y)=y+\Vert y\Vert^2\alpha.
\end{equation}
We call $\vec{v}_\alpha$ the \textit{conic vector field} given by $\alpha$.
\end{defi}

\begin{remark}
Notice that
\begin{enumerate}
 \item For every $\alpha\in\Rp^k$ we have that $v_\alpha(0)=0$.
 \item The vector field $\vec v_0$ is just the radial vector field.
 \item We can see the vector field $v_\alpha$ as a perturbation of the radial vector field $\vec v_0$.
\end{enumerate}
\end{remark}

\begin{prop}\label{prop:other.zero}
Let $\alpha\in\Rp^k$. The vector field $\vec{v}_\alpha$ has exactly two zeros, one at the origin and the other one at the point $y_\alpha=-\frac{\alpha}{\Vert\alpha\Vert^2}$.
\end{prop}

\begin{proof}
By the definition of $\vec{v}_\alpha$ given in \eqref{eq:source.vf} it is clear that the origin is a zero. Now suppose $y\neq0$ is a zero of $\vec{v}_\alpha$, again by \eqref{eq:source.vf} we
have that
\begin{equation}\label{eq:alpha.y}
\alpha=-\frac{y}{\Vert y\Vert^2}.
\end{equation}
Hence taking the norm of \ref{eq:alpha.y} we have that
\begin{equation}\label{eq:norm.alpha}
\Vert y\Vert=\frac{1}{\Vert\alpha\Vert}.
\end{equation}
Also by \eqref{eq:alpha.y} we have that $y=-\alpha\Vert y\Vert^2$ and substituting \eqref{eq:norm.alpha} we get that $y=-\frac{\alpha}{\Vert\alpha\Vert^2}$ is a zero of $\vec{v}_\alpha$.
\end{proof}

\begin{defi}
Let $\alpha\in\Rp^n$, the zero $y_\alpha$ of the conic vector fields $\vec{v}_\alpha$ given by $\alpha$ given in Proposition~\ref{prop:other.zero} is called the \textit{non-trivial zero}.
\end{defi}

Figure~\ref{fig:singularity} shows the zeros of a conic vector field.

\begin{figure}[!h]
\centering
\includegraphics[scale=0.5]{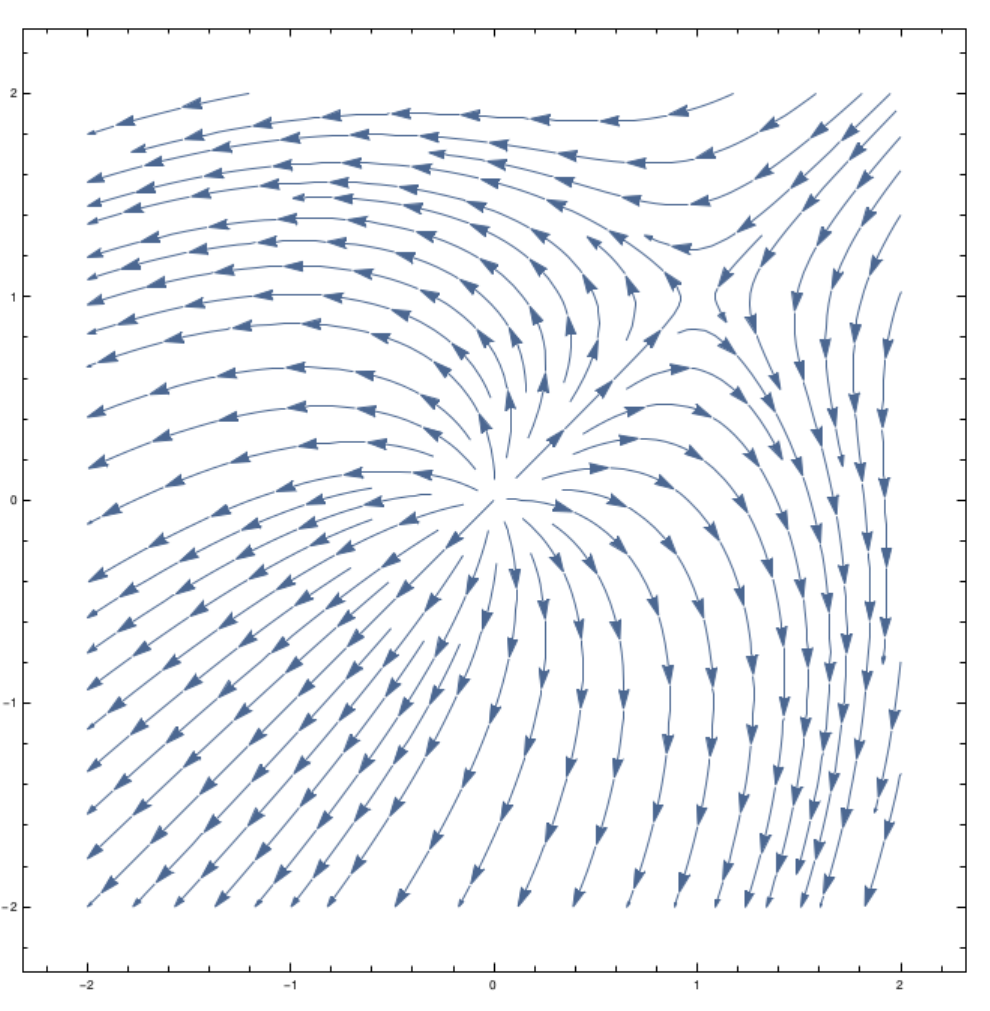}
\caption{The zeros of the  conic vector fields $\vec v_\alpha$ in $\R^2$ given by $\alpha=(-\frac{1}{2},-\frac{1}{2})$.}
\label{fig:singularity}
\end{figure}

\begin{cor}\label{cor:alpha.small}
Let $\alpha\in\Rp^k$ such that $\Vert\alpha\Vert<1$. Then in the open unit ball $\mathring{\B}_1^k$, the origin is the only zero of the conic vector fields $\vec{v}_\alpha$.
\end{cor}

\begin{proof}
By Proposition~\ref{prop:other.zero}, the non-trivial zero of $\vec{v}_\alpha$ is the point $y_\alpha=-\frac{\alpha}{\Vert\alpha\Vert^2}$ and
by \eqref{eq:norm.alpha} $\Vert y_\alpha\Vert=\frac{1}{\Vert\alpha\Vert}$, hence, if $\Vert\alpha\Vert<1$ then $\Vert y_\alpha\Vert>1$, so the only zero in the unit ball $\B_1^k\subset\R^k$ is the origin.
\end{proof}

For $\alpha \in \Rp^k$ small, $\vec v_\alpha$ is a small perturbation of the radial vector field $\vec v_0$, so near $0\in \R^k$,
the vector field $\vec v_\alpha$ is transverse to the spheres centered at $0$, as it is proved in the following proposition.

\begin{prop}\label{prop:source.tranverse}
Let $\alpha\in\Rp^k$ with $\Vert\alpha\Vert<1$. Then the conic vector fields $\vec{v}_\alpha$ given by $\alpha$ is transverse to all the spheres $\BS_\eta^{k-1}\subset\R^k$ centered at $0$ of radius $0<\eta<1$.
\end{prop}

\begin{proof}
The vector $\vec{v}_\alpha(y)$ is tangent to the sphere which passes through $y$ if and only if 
$0=\langle \vec{v}_\alpha(y),y\rangle=\langle y+\Vert y\Vert^2\alpha,y\rangle=\Vert y\Vert^2\bigl(1+\langle\alpha,y\rangle\bigr)$.
Thus, if $y\neq0$, we have that $\langle \vec{v}_\alpha(y),y\rangle=0$ if and only if
\begin{equation}\label{eq:inp.a.y}
\langle\alpha,y\rangle=-1.
\end{equation}
Clearly the non-trivial zero of $\vec{v}_\alpha$ (which is a negative multiple of $\alpha$) satisfies \eqref{eq:inp.a.y}
\begin{equation*}
\langle\alpha,y_\alpha\rangle=\langle\alpha,-\frac{\alpha}{\Vert\alpha\Vert^2}\rangle=-1.
\end{equation*}
Let $H_\alpha$ be the subspace orthogonal to the line generated by $\alpha$, i.~e.,
\begin{equation*}
 H_\alpha=\{w\in\R^k\,|\,\langle\alpha,w\rangle=0\}.
\end{equation*}
Hence, any vector of the form
\begin{equation}\label{eq:tangent.vectors}
y_\alpha+w,\quad \text{with $w\in H_\alpha$,}
\end{equation}
satisfies \eqref{eq:inp.a.y} since $\langle\alpha,y_\alpha+w\rangle=\langle\alpha,y_\alpha\rangle+\langle\alpha,w\rangle=-1$.
The vector of this form with smallest norm is precisely $y_\alpha$, which by Corollary~\ref{cor:alpha.small} is outside the closed unit ball $\B_1^k$ since $\Vert\alpha\Vert<1$.
Hence, any other vector of the form \eqref{eq:tangent.vectors} is also outside the ball $\B_1^k$ and therefore the conic vector fields given by $\alpha$ is transverse to all
the spheres $\BS_\eta^{k-1}$ centered at $0$ of radius $0<\eta<1$.
\end{proof}

\begin{remark}
In view of Corollary~\ref{cor:alpha.small} and Proposition~\ref{prop:source.tranverse}, from now on, we will only consider $\alpha\in\Rp^k$ such that $\Vert\alpha\Vert<1$, i.~e., we only consider $\alpha$ in the open unit ball $\underline{\mathring{\B}}_1^k$ in the parameter space, and we shall restrict the conic vector field $\vec{v}_\alpha$ to the open unit ball $\mathring{\B}_1^k\subset\R^k$.
\end{remark}

\begin{lemma}\label{lem:away}
Let $\alpha\in \underline{\mathring{\B}}_1^k$ and consider the conic vector field $\vec{v}_\alpha(y)$. Then we have
\begin{equation*}
\langle \vec{v}_\alpha(y),y\rangle>0,\quad\text{for every $y\in\mathring{\B}_1^k$.}
\end{equation*}
\end{lemma}

\begin{proof}
We have that
\begin{equation*}
\langle \vec{v}_\alpha(y),y\rangle=\langle y+\Vert y\Vert^2\alpha,y\rangle=\Vert y\Vert^2\bigl(1+\langle\alpha,y\rangle\bigr)=\Vert y\Vert^2\bigl(1+\Vert\alpha\Vert\Vert y\Vert\cos\theta\bigr)
\end{equation*}
where $\theta$ is the angle between $\alpha$ and $y$ in $\R^k$. Since $\Vert\alpha\Vert<1$ and $\Vert y\Vert<1$ we have that $\bigl\vert\Vert\alpha\Vert\Vert y\Vert\cos\theta\bigr\vert<1$, so the expression between parenthesis is always positive. Thus $\langle \vec{v}_\alpha(y),y\rangle>0$.
\end{proof}

\begin{remark}
Lemma~\ref{lem:away} says that for $\alpha\in \underline{\mathring{\B}}_1^k$ the vector $\vec{v}_\alpha(y)$ points ``away'' from $0\in\R^k$ for all $y\in\mathring{\B}_1^k$.
\end{remark}

Figure~\ref{fig:curves} shows the conic vector field  of Figure~\ref{fig:singularity} but restricted to the ball $\mathring{\B}_{\frac{1}{2}}^k$ where it has only one zero at the origin and is pointing away from it.

\begin{figure}[!h]
\centering
\includegraphics[scale=0.5]{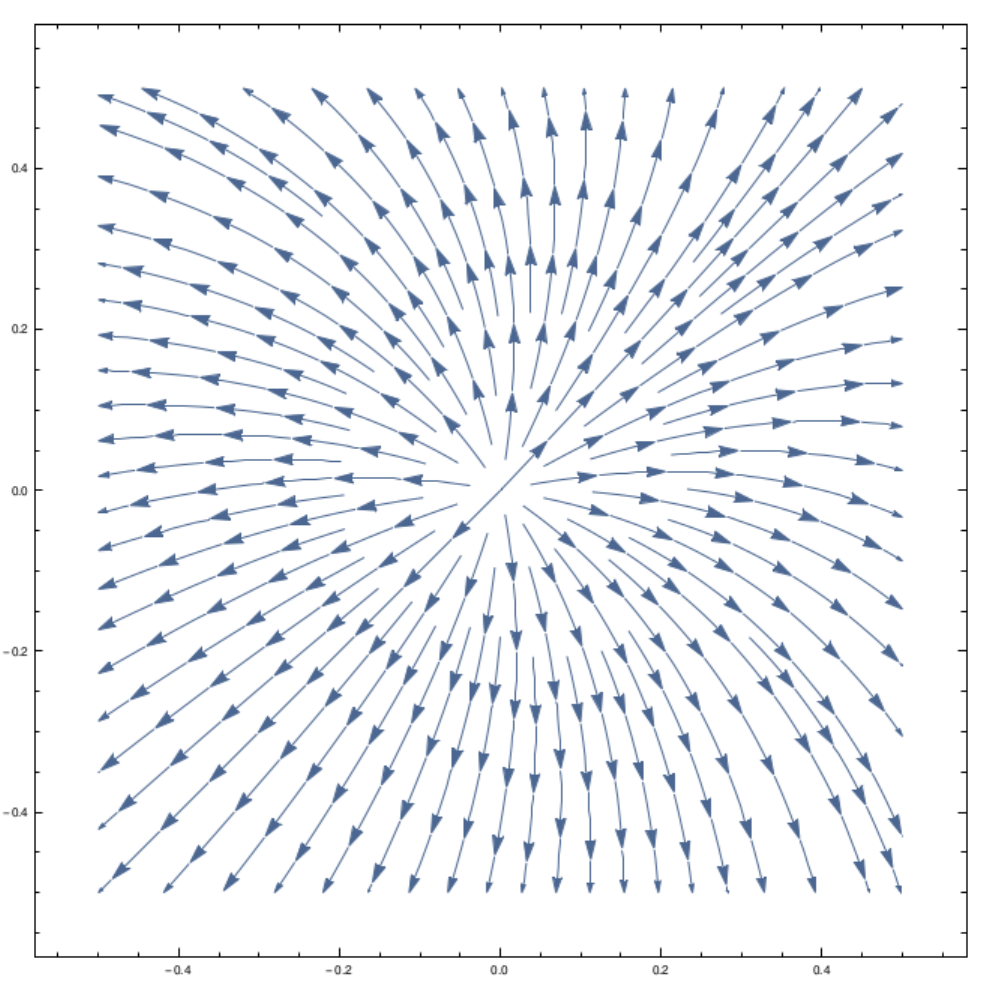}
\caption{The conic vector fields $\vec v_\alpha$ in $\R^2$ given by $\alpha=(-\frac{1}{2},-\frac{1}{2})$ on $\mathring{\B}_{\frac{1}{2}}^k$.}
\label{fig:curves}
\end{figure}

Now let us see that the conic vector field $\vec{v}_\alpha$ defines a conic homeomorphism $h_\alpha\colon\B_\eta^k\to\B_\eta^k$ on any closed ball $\B_\eta^k$ of radius $0<\eta<1$ centred at $0\in\R^k$  (compare with \cite[page~20]{Milnor:SPCH}).

Let $\alpha\in \underline{\mathring{\B}}_1^k$ and consider the conic vector field $\vec{v}_\alpha$ defined on the open unit ball $\mathring{\B}_1^k\subset\R^k$. Normalize $\vec{v}_\alpha$ by setting
\begin{equation}
\hat{v}_\alpha(y)=\frac{\vec{v}_\alpha(y)}{\langle2y,\vec{v}_\alpha(y)\rangle}
\end{equation}
and consider the integral curves of $\hat{v}_\alpha$, that is, the smooth solution curves $y=p(t)$, defined say for $a<t<b$, of the differential equation
\begin{equation}\label{eq:ODE}
\frac{dp(t)}{dt}=\hat{v}_\alpha(p(t)).
\end{equation}
Let $r\colon\R^k\to\R$ be the ``norm squared function'' $r(y)=\Vert y\Vert^2$. Given any solution $p(t)$, the derivative of the composition $r\circ p\colon(a,b)\to\R$ is given by
\begin{equation*}\label{eq:drp=1}
\frac{d(r\circ p)}{dt}=\langle2y,\hat{v}_\alpha(y)\rangle=\frac{\langle2y,\vec{v}_\alpha(y)\rangle}{\langle2y,\vec{v}_\alpha(y)\rangle}=1
\end{equation*}
where $y=p(t)$. So we have $r(p(t))=t+\text{constant}$. Thus, subtracting a constant from the parameter $t$ if necessary, we may suppose that
\begin{equation}\label{eq:rpt=t}
r(p(t))=\Vert p(t)\Vert^2=t.
\end{equation}
This solution can be extended through the interval $(0,\eta^2]$ (see \cite[page~20]{Milnor:SPCH}). Note that the integral curve $p(t)$ is uniquely determined by the initial value $p(\eta^2)=\theta\in\BS_\eta^{k-1}$. For each $\theta\in\BS_\eta^{k-1}$ denote by
\begin{equation*}
 P(\theta,t),\quad 0<t\leq\eta^2,
\end{equation*}
the unique integral curve which satisfies the initial condition
\begin{equation*}
 P(\theta,\eta^2)=\theta.
\end{equation*}
Let $\CC_\theta^\alpha=P\bigl(\theta,(0,\eta^2)\bigr)$, i.~e., the image of the integral curve $P(\theta,t)$.
We have that the function
\begin{equation*}
P\colon\BS_\eta^{k-1}\times(0,\eta^2]\to\B_\eta^k\setminus\{0\}
\end{equation*}
is a diffeomorphism.
Finally, note that $P(\theta,t)$ tends uniformly to $0\in\BS_\eta^{k-1}$ as $t\to0$. Therefore the correspondence
\begin{equation*}
 t\theta\to P(\theta,t\eta^2)
\end{equation*}
defined for $0<t\leq 1$, extends uniquely to a conic homeomorphism
\begin{equation}
h_\alpha\colon\B_\eta^k\to\B_\eta^k
\end{equation}
since outside the origin is a diffeomorphism and sends the open interval $\LL_\theta$ to the curve $\CC_\theta^\alpha$.

These conic homeomorphisms will be used in the next section to prove the existence of suitable parameters for any map satisfying the transversality property.

\section{Suitable vector fields}

Let $f\colon (\R^n,0) \to (\R^k,0)$, with $2 \leq k \leq n$, be a locally surjective analytic map with an isolated critical value at $0\in\R^k$ satisfying the transversality property.

\begin{defi}
Let $x\in\R^n\setminus f^{-1}(0)$, let $N_x(f^{-1}(f(x)))$ be the normal subspace to the tangent space $T_x(f^{-1}(f(x)))$ of the fibre of $f$ which contains $x$.
Given a vector field $\vec{v}$ in $\R^k\setminus \{0\}$ the lifting $\vec{u}$ given by the isomorphism
\begin{equation*}
(Df_x)_|: N_x \big(f^{-1}(f(x)) \big) \to T_{f(x)} \big( \R^k \big),
\end{equation*}
is called \textit{the canonical $f$-lifting} of $\vec{v}$.
\end{defi}

For each $\alpha \in \underline{\mathring{\B}}_1^k$ let $\vec{u}_\alpha$ be the canonical $f$-lifting of the conic vector field $\vec v_\alpha$.

\begin{defi}
A vector field $\vec{w}$ on $\R^n\setminus f^{-1}(0)$ is called an \textit{$f$-vertical vector field} if $\vec{w}(x)\in T_x(f^{-1}(f(x)))$ for every $x\in\R^n\setminus f^{-1}(0)$.
\end{defi}

\begin{defi}
Let $\vec{v}$ be a vector field on $\R^k\setminus \{0\}$ and let $\vec{u}$ be its canonical $f$-lifting. Given a $f$-vertical vector field $\vec{w}$ on $\R^n\setminus f^{-1}(0)$, the vector field given by
\begin{equation*}
 \vec{u}^{\vec{w}}(x):=\vec{u}(x)+\vec{w}(x)
\end{equation*}
is a $f$-lifting of $\vec{v}$, which we call \textit{the $f$-lifting of $\vec{v}$ determined by $\vec{w}$}.
\end{defi}

\begin{defi} \label{defi_1}
By \cite[Theorem~5.1, Remark~5.7]{Cisneros-Seade-Snoussi:d-regular} $f$ has a Milnor-Lê fibration \eqref{eq:RMLF} on a Milnor tube $T_{\varepsilon,\delta}(f)$ with $0<\delta\ll\varepsilon$.
\begin{itemize}
\item[$(i)$] We say that $\alpha \in \underline{\mathring{\B}}_1^k$ is a \textit{suitable parameter for $f$ at a point $p \in \R^n$} if there is a lifting $\vec{u}'_\alpha$ of $\vec{v}_\alpha$ such that $\langle \vec{u}'_\alpha(p),p  \rangle \neq 0$;
\item[$(ii)$] We say that $\alpha \in \underline{\mathring{\B}}_1^k$ is a \textit{suitable parameter for $f$ in a subset $W \subset \R^n$} if there is a lifting $\vec{u}'_\alpha$ of $\vec{v}_\alpha$ on $W$ such that $\langle \vec{u}'_\alpha(p),p  \rangle \neq 0$ for every $p \in W$, that is, $\alpha$ is a suitable parameter for $f$ at $p$ for every $p \in W$ with respect to the same lifting $\vec{u}'_\alpha$ of $\vec{v}_\alpha$;
\item[$(iii)$] We say that $\alpha \in \underline{\mathring{\B}}_1^k$ is a \textit{suitable parameter for $f$} if for every $p\in\B_\varepsilon^n \setminus \mathring{T}_{\varepsilon,\delta}(f)$, outside the open solid Milnor tube
\begin{equation*}
\mathring{T}_{\varepsilon,\delta}(f) := \B_\varepsilon^n \cap f^{-1}(\mathring{\B}_\delta^k),
\end{equation*}
there exist an open neighbourhood $W_p$ of $p$ where $\alpha$ is a suitable parameter for $f$ in $W_p$.
In this case, we say that the corresponding vector field $\vec v_\alpha$ in $\B_\eta^k$ is a \textit{suitable vector field for $f$}.
\end{itemize}
\end{defi}

By the definition of a suitable parameter for $f$ we have:

\begin{prop} \label{theo_main}
Let $f\colon (\R^n,0) \to (\R^k,0)$, with $2 \leq k \leq n$, be an analytic map-germ with isolated critical value. If $\alpha \in \Rp^k$ is a suitable parameter for $f$, then $f$ is $d_{h_\alpha}$-regular (outside $\mathring{T}_{\varepsilon,\delta}(f)$) with respect to the conic homeomorphism $h_\alpha$ given by the conic vector field $\vec{v}_\alpha$.
\end{prop}

\begin{proof}
Since $f\colon (\R^n,0) \to (\R^k,0)$ has isolated critical value at $0\in\R^k$ we have that the set $$E_\theta := f^{-1}(\CC_\theta^\alpha)$$ is a submanifold of $\R^n\setminus f^{-1}(0)$.
Since $\alpha \in \Rp^k$ is a suitable parameter for $f$, for every $p\in\B_\varepsilon^n \setminus \mathring{T}_{\varepsilon,\delta}(f)$ there exist an open neighbourhood $W_p$ of $p$ an a lifting $\vec{u}'_\alpha$ of $\vec{v}_\alpha$ on $W$ such that $\langle \vec{u}'_\alpha(p),p  \rangle \neq 0$ for every $p \in W$.
This implies that $E_\theta$ intersects the spheres $\BS_{\varepsilon'}^{n-1}$ transversely in $\B_\varepsilon^n \setminus \mathring{T}_{\varepsilon,\delta}(f)$, for every $0<\varepsilon' \leq \varepsilon$.
In other words, $f$ is $d_h$-regular.
\end{proof}

Now the goal is to prove that suitable parameters for $f$ always exist.

\begin{lemma}\label{lem:balls}
Let $p\in \B_\varepsilon^n \setminus \mathring{T}_{\varepsilon,\delta}(f)$ and let $\vec{w}$ be a $f$-vertical vector field on an open neighbourhood $W'_p$ of $p$ in $\B_\e^n \setminus \mathring{T}_{\varepsilon,\delta}(f)$. Suppose that $\tilde{\alpha}\in\underline{\mathring{\B}}_1^k$ is such that the $f$-lifting $\vec{u}_{\tilde{\alpha}}^{\vec{w}}$ determined by $\vec{w}$ is transverse to $\BS_{\|p\|}^{n-1}$ at $p$, with $\langle \vec{u}_{\tilde{\alpha}}^{\vec{w}}(p),p  \rangle > 0$ (respectively $\langle \vec{u}_{\tilde{\alpha}}^{\vec{w}}(p),p  \rangle < 0$). Then there exists a small neighbourhood $W_p\subset W'_p$ of $p$ and a real number $\nu_p>0$ such that for every $\alpha$ in the ball $\underline{\mathring{\B}}^k_{\nu_p}({\tilde{\alpha}})$ in $\underline{\mathring{\B}}_1^k$ of radius $\nu_p$ centered at $\tilde{\alpha}$, and for every $x\in W_p$, one has that $\langle \vec{u}_\alpha^{\vec{w}}(x),x  \rangle > 0$ (respectively $\langle \vec{u}_\alpha^{\vec{w}}(x),x  \rangle < 0$).
\end{lemma}

\begin{proof}
Define the map
\begin{align*}
g_p\colon W'_p \times \underline{\mathring{\B}}_1^k &\longrightarrow \R,\\
(x,\alpha) &\longmapsto \langle \vec{u}_\alpha^{\vec{w}}(x),x  \rangle.
\end{align*}
By hypothesis $g_p(p,\tilde{\alpha}) = \langle \vec{u}_{\tilde{\alpha}}^{\vec{w}}(p),p  \rangle=c$, with $c>0$ (respectively $c<0$). Hence the set $U:=g_p^{-1}\big( ]0,2c[ \big)$ 
(respectively $U:=g_p^{-1}\big( ]2c,0[ \big)$) is open in $W'_p \times \mathring{\B}_1^k$
and contains the point $(p,\tilde{\alpha})$. Thus, there exist an open neighborhood $W_p(\tilde{\alpha})$ of $p$ in $W'_p$ and an open ball $\underline{\mathring{\B}}_{\nu_p(\tilde{\alpha})}^k(\tilde{\alpha})$ centered at
$\tilde{\alpha} \in \underline{\mathring{\B}}_1^k$ such that $(p,\tilde{\alpha}) \in W_p(\tilde{\alpha}) \times \B_{\nu_p(\tilde{\alpha})}^k(\tilde{\alpha}) \subset U$.
Therefore $\langle \vec u_\alpha (x),x  \rangle >0$ (respectively $\langle \vec u_\alpha (x),x  \rangle <0$) for every $x \in W_p(\tilde{\alpha})$ and for every $\alpha \in \B_{\nu_p(\tilde{\alpha})}(\tilde{\alpha})$.
\end{proof}

\begin{remark}
Notice that in Lemma~\ref{lem:balls} if we take the vertical vector field $\vec{w}$ to be the zero-vector field, the statement specializes to the canonical liftings $\vec{u}_\alpha$ of
the vector fields $\vec{v}_\alpha$.
\end{remark}

In other words, Lemma~\ref{lem:balls} says:
\begin{cor}\label{cor:balls}
If $\tilde{\alpha}$ is a suitable parameter for $f$ at a point $p \in \B_\varepsilon^n \setminus \mathring{T}_{\varepsilon,\delta}(f)$ then there exist a neighborhood $W_p(\tilde{\alpha})$ of $p$ in $\B_\varepsilon^n \setminus \mathring{T}_{\varepsilon,\delta}(f)$ and a radius $\nu_p(\tilde{\alpha})>0$ such that every $\alpha \in \underline{\mathring{\B}}_{\nu_p(\tilde{\alpha})}^k (\tilde{\alpha})$ is a suitable parameter for $f$ in $W_p(\tilde{\alpha})$.
\end{cor}

We have:

\begin{prop} \label{prop_1}
Let $f\colon (\R^n,0) \to (\R^k,0)$, with $2 \leq k \leq n$, be a locally surjective analytic map-germ with isolated critical value. Then there exist a real number $\omega$, with $0<\omega< 1$
such that every $\alpha \in \underline{\mathring{\B}}_\omega^k\subset\underline{\mathring{\B}}_1^k$ is a suitable parameter for $f$.
\end{prop}

\begin{proof}
First we prove that for every $p \in \B_\varepsilon^n \setminus \mathring{T}_{\varepsilon,\delta}(f)$ there exist an open neighborhood $W_p$ of $p$ in $\B_\varepsilon^n \setminus \mathring{T}_{\varepsilon,\delta}(f)$,
and an open ball $\underline{\mathring{\B}}_{\nu_p}^k$ in the unit ball $\underline{\mathring{\B}}_1^k$ such that every $\alpha \in \underline{\mathring{\B}}_{\nu_p}^k$ is a suitable parameter for $f$ in $W_p$.

Fix $p \in \B_\varepsilon^n \setminus \mathring{T}_{\varepsilon,\delta}(f)$ and let $\vec u_0 (p)$ be the canonical lifting of the radial vector field $\vec v_0 (p)$. We have two cases:

\paragraph{\bf Case 1:} Suppose $\langle \vec u_0 (p),p  \rangle \neq 0$. Set $\tilde{\alpha}=0$ and $\vec{w}=\vec{0}$, the zero vector field.
Then by Lemma~\ref{lem:balls} there exist a neighbourhood $W_p$ of $p$ and $\nu_p>0$ such that $\langle \vec{u}_\alpha (x),x  \rangle \neq0$ for every $x \in W_p$ and every $\alpha \in \underline{\mathring{\B}}_{\nu_p}^k$.

\paragraph{\bf Case 2:} Suppose $\langle \vec u_0 (p),p  \rangle =0$. This means that the vector $\vec u_0 (p)$ is tangent to the sphere $\BS_{\Vert p\Vert}^{k-1}$. This implies that the fibre $f^{-1}(f(p))$ which contains $p$ and the sphere $\BS_{\Vert p\Vert}^{k-1}$ are transverse at $p$. Since transversality is a stable condition, there exist an open neighbourhood $W'_p$ of $p$ such that for all $x\in W'_p$, the fibre $f^{-1}(f(x))$ which contains $x$ and the sphere $\BS_{\Vert x\Vert}^{k-1}$ are transverse at $x$. We construct an $f$-vertical vector field on $W'_p$ as follows. For every $x\in W'_p$ let $\vec{w}'(x)=x$ be the restriction to $W'_p$ of the radial vector field on $\R^n$. Let $\vec{w}(x)$ be the projection of $\vec{w}'(x)=x\in T_x\R^n$ onto the tangent space $T_x(f^{-1}(f(x)))$ at $x$ of the fibre which contains $x$.
We have that $\vec{w}$ is a non-zero $f$-vertical vector field because of the transversality between fibres and spheres in $W'_p$ and also $\langle \vec{w}(x),x  \rangle > 0$ for every $x\in W'_p$.
Hence, the $f$-lifting $\vec{u}_0^{\vec{w}}$ of $\vec{v}_0$ determined by $\vec{w}$ satisfies $\langle\vec{u}_0^{\vec{w}}(x),x\rangle> 0$ for every $x\in W'_p$.
Then by Lemma~\ref{lem:balls} there exist a neighbourhood $W_p$ of $p$ and $\nu_p>0$ such that $\langle \vec{u}_\alpha^{\vec{w}}(x),x  \rangle \neq0$ for every $x \in W_p$ and every $\alpha \in \underline{\mathring{\B}}_{\nu_p}^k$.

The open neighbourhoods $\{W_p\}$ form an open cover of the compact set $\B_\varepsilon^n \setminus \mathring{T}_{\varepsilon,\delta}(f)$, so there is an open subcover $\{W_{p_1},\dots,W_{p_m}\}$.
Set  $\underline{\mathring{\B}}_\omega^k=\bigcap_{i=1}^m \underline{\mathring{\B}}_{\nu_{p_i}}^k$, that is $ \omega=\min_{i=1}^m\{\nu_{p_i}\}$. Thus every $\alpha\in\underline{\mathring{\B}}_\omega^k$ is a suitable parameter for $f$.
\end{proof}

Now we can state the main theorem.
By Proposition~\ref{prop_1} $\vec v_\alpha$ is a suitable vector field for $f$ for every $\alpha\in\underline{\mathring{\B}}_\omega^k$.
By Proposition~\ref{theo_main} $f$ is $d_h$-regular with respect to the conic homeomorphism $h_\alpha$ given by the conic vector field $\vec v_\alpha$.
As an immediate consequence of Theorem \ref{prop_principal}, we have:

\begin{theo}\label{thm:mainthm}
Let $f\colon (\R^n,0) \to (\R^k,0)$, with $2 \leq k \leq n$, be a locally surjective analytic map-germ with an isolated critical value at $0\in\R^k$ which satisfies the transversality property.
Then there exists a real number $0<\omega< 1$ such that for every $\alpha \in \underline{\mathring{\B}}_\omega^k$ the Milnor-L\^e fibration in the tube
\begin{equation*}
 f\colon T_{\varepsilon,\delta}(f)\to\BS_\delta^{k-1},
\end{equation*}
is equivalent to the Milnor fibration on the sphere:
\begin{equation*}
\frac{h_\alpha^{-1} \circ f}{\|h_\alpha^{-1} \circ f\|}\colon\BS_\varepsilon^{n-1} \setminus f^{-1}(0) \to \BS^{k-1}
\end{equation*}
where $h_\alpha: \B_\eta^k \to \B_\eta^k$ is the conic homeomorphism defined by the conic vector field $\vec{v}_\alpha$.
\end{theo}

The Theorem stated in the Introduction follows from Theorem~\ref{thm:mainthm}.

\section*{Acknowledgments}
We thank Prof. José Seade for many enlightening conversations.

\bibliography{mypapers,mymatbk,matart}
\bibliographystyle{plain}

\end{document}